\documentclass{mathscan}

\usepackage{mathrsfs}
\usepackage{verbatim}
\usepackage{tikz}
\usepackage{subcaption}
\def\be {\begin{equation}}
\def\en{\end{equation}}

\newcommand\tb{}
\newcommand\tr{\textcolor{red}}

\newtheorem{theorem}{Theorem}[section] 
\newtheorem{lemma}[theorem]{Lemma}     
\newtheorem{corollary}[theorem]{Corollary}
\newtheorem{proposition}[theorem]{Proposition}

\theoremstyle{definition}
\newtheorem{definition}[theorem]{Definition}
\newtheorem{example}[theorem]{Example}
\newtheorem*{ack*}{Acknowledgment}

\theoremstyle{remark}
\newtheorem{remark}[theorem]{Remark}

\newtheoremstyle{condition}
{3pt}
{3pt}
{\rm}
{0pt}
{\itshape}
{: }
{0pt}
{\thmname{#1}\thmnumber{#2}\thmnote{(#3)}}

\theoremstyle{condition}
\newtheorem{condition}{C}

\numberwithin{equation}{section}

\newcommand{\N}{\mathbb{N}}

\begin{document}
\begin{abstract}
	We develop conditions for the coding of a Bratteli-Vershik system according to initial path segments to be periodic,
	equivalently for a constructive symbolic recursive scheme corresponding to a cutting and stacking process to produce a periodic sequence.
	This is a step toward understanding when a Bratteli-Vershik system can be essentially faithfully
	represented by means of a natural coding as a subshift on a finite alphabet.
\end{abstract}
\title{Periodic codings of Bratteli-Vershik systems}
\author{Sarah Frick}
\address{Department of Mathematics, Furman University, Greenville, SC 29613 USA}
\email{sarah.frick@furman.edu}
\author{Karl Petersen}
\address{Department of Mathematics,
	CB 3250 Phillips Hall,
	University of North Carolina,
	Chapel Hill, NC 27599 USA}
\email{petersen@math.unc.edu}
\author{Sandi Shields}
\address{College of Charleston, 66 George St., Charleston, SC 29424-0001 USA}
\email{shieldss@cofc.edu}

\maketitle

\section{Introduction}
Cutting and stacking constructions, equivalently Bratteli-Vershik systems, have been used to construct many of the important examples in ergodic theory and to study the fundamental properties of classes of measure-preserving systems. Indeed, (an isomorphic copy of) every measure-preserving system can be presented by such a construction.
Similar statements apply to topological dynamical systems on the Cantor set.
There are generalizations to actions by other groups, infinite measure systems, and nonsingular actions--see, for example, \cite{Danilenko2016,Danilenko2004,DaiSilva2015, Eigen2014}.
We are interested in finding conditions for representing a system that is presented in this way as a subshift on a (usually finite) alphabet.
A transformation defined by cutting and stacking, or as the Vershik map on a Bratteli diagram, might not be continuous and might not be defined everywhere.
If one can explicitly produce a subshift that is measure-theoretically isomorphic to the original system, then the resulting system, which consists of a continuous map on a compact metric space, will provide a setting in which questions about topological dynamical properties (such as mixing) and combinatorial properties (such as complexity) will make sense.

A system is of {\em rank one} if it is measure-theoretically isomorphic to a Lebesgue-measure-preserving map on a finite interval that is defined by a cutting and stacking process in which at each stage there is a single tower and another tower of height one (the ``reservoir" of ``spacers"). See \cite{Ferenczi1997} for equivalent conditions, for example that the partitions consisting of the levels of Rokhlin towers generate the full sigma-algebra.
A {\em constructive symbolic rank one} presentation of a system represents it isomorphically as a subshift.
The conditions listed in \cite{Ferenczi1997} for a system to be rank one were known to be equivalent, with the possible exception that for odometers the symbolic sequence defined by the cutting and stacking recursion is periodic, so that the naturally associated subshift is finite, while the odometer itself is infinite.
The paper \cite{AFP} completed the equivalence of the definitions of rank one by producing for every rank one system, including in particular the $2$-odometer, a constructive symbolic rank one representation, thus a subshift measure-theoretically isomorphic to the given rank one system.
Previously, Kalikow \cite{Kalikow1984} proved that if the symbolic sequence naturally associated with a rank one cutting and stacking procedure is not periodic, then the original system is isomorphic to the subshift defined by the sequence.
El Abdaloui, Lemanczyk, and de la Rue \cite[Lemma 10]{ElAbdaloui2014} showed that, on the other hand, if the sequence associated with the rank one cutting and stacking is periodic, then the system is isomorphic to an odometer.
We extend this result in Theorem \ref{thm:rankone} below.
Foreman and Weiss \cite{ForemanWeiss2015, ForemanWeiss2016} studied the structure with respect to joinings of the class of systems with odometer factors and its relevance for the realization and classification problems in the theory of dynamical systems.
For history and other work on these questions, see the discussions in \cite{AFP} and the references given there.

A system defined by cutting and stacking has for each $k \geq 1$ a natural partition $\gamma_k$ defined by the levels of the towers at stage $k$ of the construction.
Similarly, a Bratteli-Vershik system has for each $k \geq 1$ a partition $\alpha_k$ defined by the cylinder sets corresponding to initial segments of $k$ edges starting at the root (details are given below).
These partitions correspond to measure-preserving factors of the system, which we call  {\em $k$-factors}, or {\em level-$k$ factors}, as in \cite{dm2008}.
	(The image of the orbit of a path may be called a $k$-coding of the path, or of its orbit.)
We say that a measure-preserving system defined in one of these two ways is {\em essentially $k$-expansive} if the partition $\gamma_k$ (or $\alpha_k$) generates the full sigma-algebra under the action of the transformation, up to sets of measure $0$.
An essentially $k$-expansive system is measure-theoretically isomorphic to its $k$-factor subshift with the push-forward invariant measure under the factor mapping.
Our main goal is to determine when a system is essentially $k$-expansive for some $k$.
For the Pascal system with the usual left-right ordering of incoming edges, essential $1$-expansiveness was proved in \cite{xman}; see also \cite{MelaPetersen2005} and a similar result for classes of systems in \cite{Mela2006, Frick2009}.
Essential $3$-expansiveness for the Pascal system with any ordering of the edges was proved in \cite{FPS2017}.

Downarowicz and Maass \cite{dm2008} showed that the topological dynamical system defined by a simple, properly ordered Bratteli-Vershik diagram with a bounded number of vertices per level, which they called {\em topological finite rank} but which we prefer to call {\em bounded width}, is either (topologically) expansive (meaning that for some $k$ the partition $\alpha_k$ generates the topology) or else the system is topologically conjugate to an odometer.
In \cite{FPS2017} we gave necessary and sufficient conditions for a simple properly ordered Bratteli-Vershik system to be topologically conjugate to an odometer: there should be infinitely many {\em uniformly ordered} levels.
Below we extend this condition to more general diagrams.

In view of these remarks, an obvious obstruction to essential $k$-expansiveness of an infinite system is that the $k$-factor is finite; for an odometer presented in the usual way, this happens for every $k$.
Therefore, in order to extend the above-mentioned results, we investigate for Bratteli-Vershik systems whether some or all $k$-factors are finite and whether such a system is isomorphic to an odometer.
Note that \cite{AFP}  presents a Bratteli-Vershik system isomorphic to an odometer {\em and} to each of its $k$-factors, $k \geq 1$.
On the other hand, an infinite entropy, uniquely ergodic, simple, properly ordered Bratteli-Vershik system (such exist by the \tb{Jewett-Krieger \cite{Jewett1970,Krieger1972} and Herman-Putnam-Skau \cite{HPS1992}} theorems) {\em cannot} be isomorphic to any of its $k$-factors.
Any Bratteli-Vershik system with a finite ergodic measure is measure-theoretically isomorphic to the inverse limit of its $k$-factors.
So if every $k$-factor is finite, then the system is isomorphic to an odometer or a permutation of finitely many points.

In Theorem \ref{thm:rankone} we reprove and extend the result from \cite{ElAbdaloui2014} to show that for a Bratteli-Vershik system in the standard rank one form from \cite{AFP}, if some $k$-factor is finite then the diagram must have a special form, every $k$-factor is finite, and the system is measure-theoretically isomorphic to an odometer.
Theorem \ref{thm:general} extends this result to a class of more general cutting and stacking procedures and associated diagrams. In this class there are two minimal paths, one of which is a fixed point of the transformation.
In addition, in the cutting and stacking presentation the base of the tower does not come from a spacer reservoir.
We then show that within this class some $k$-factor is finite if and only if eventually all levels of the diagram satisfy what we call the {\em local deficit condition with respect to level $k$}.
This condition, which extends the definition of {\em uniformly ordered} in \cite{FPS2017},
imposes a structure on the partial ordering of the edges and how the spacers are interspersed,
producing a sort of local incomplete periodicity (see Definitions \ref{def:locdeficit} and \ref{def:semi}).
Section \ref{sec:examples} presents examples of the various possibilities for $k$-factors.

\section{Setup and notation}\label{sec:setup}
Throughout, we shall assume that $\mathcal B=(\mathcal{V},\mathcal{E},\leq)$ is an ordered Bratteli diagram with a finite number $K_n+1$ of vertices $v(n,j), j=1,\dots, K_n+1$, at every level $n=0,1,\dots$;
$K_0=0$, so that at level $0$ there is just a single  vertex, $v(0,1)$, called the {\em root};
every vertex is the source of at least one edge;
and every vertex except the root is the range of at least one edge.
\tb{$\mathcal E_n$ denotes the set of directed edges from level $n$ to level $n+1$, $n \geq 0$.}
(As with any Bratteli diagram, the range function $r:\mathcal{E}\to \mathcal{V}$ maps an edge to its terminal vertex, and the source function $s:\mathcal{E}\to \mathcal{V}$ maps an edge to its source.)
 There is a partial edge ordering $\xi:\mathcal{E}\to \N$ which gives a total ordering on edges with the same range and extends to a partial ordering of infinite paths in the usual way.  
 $X$ is the space of infinite directed paths $x=x_0 x_1 \dots$, each $x_i \in \mathcal E_i$,
beginning at the root, and $T$ is the adic (or successor, or Vershik) transformation defined on the set of non-maximal paths. 
 Similarly, $T^{-1}$ is defined on the set of non-minimal paths.
Two paths $x$ and $y$ are in the same orbit if they are tail equivalent, meaning there exists an $N\in\mathbb{N}$ such that for all $n\geq N$, $x_n=y_n$. 
See, for example, \cite{Durand2010, MelaPetersen2005, BezuglyiKarpel2016} for background.

\tb{In order to model classical cutting and stacking constructions in ergodic theory, we assume 
\begin{condition}\label{cond:cond1}
for all $n \geq 1$ we have all $K_n \geq 1$, so that there are at least two vertices at each level after the root, and 
\end{condition}
\begin{condition}\label{cond:cond2}
 each {\em spacer vertex} $v(n, K_n+1)$ has a single incoming edge from vertex $v(n-1, K_{n-1}+1)$ (in other words $|r^{-1}(v(n,K_n+1))|=1$ and $s(r^{-1}(v(n,K_n+1))=v(n-1,K_{n-1}+1)$.
 \end{condition}
}
 There is a unique infinite path that passes through $v(n,K_n+1)$ for all $n>0$. We denote this path by $x_s$ and define $T(x_s)=x_s$.
 We note that if for all large enough $n$ we have $|s^{-1}v(n,K_n+1)|=1$ (so that eventually each spacer vertex has only one outgoing edge),
 then $x_s$ is an isolated path. Because we allow the case when $x_s$ is isolated, this setup is not very restrictive.
 These systems are not {\em aperiodic} (see \cite{Medynets2006, BezuglyiKwiatkowskiMedynets2009, BezuglyiKwiatkowskiMedynetsSolomyak2010, BezuglyiKwiatkowskiMedynetsSolomyak2013}), but they can be {\em almost simple} (see \cite{Danilenko2001, Yuasa2007}).

\begin{definition}\label{def:pseudocomplete}
  We say that level $n$ is \emph{pseudo-complete} if for every $j\in \{1,\dots,K_{n-1}\}$ and $j'\in \{1,\dots,K_n\}$ there is at least one edge connecting $v(n-1,j)$ and $v(n,j')$.\end{definition}

 Denote by $\dim(n,j)$ the number of segments from the root to $v(n,j)$. We code infinite paths starting at the root and their orbits under the adic transformation $T$ according to their first $k$ edges after the root:
if there are $d_k$ segments from the root to level $k$, assign the symbols from the alphabet $A_k=\{0,...,d_k-2,s_k\}$ to the segments from the root to vertices at level $k$, where the unique segment from the root to vertex $(k,K_k+1)$ is given the label $s_k$,
which can be thought of as a ``spacer''.
This produces a (possibly many-to-one) measurable map $\phi_k$ from $X$ to $A_k^\mathbb{N}$ as follows.
Denote by $\alpha_k$ the partition of $X$ into the cylinder sets $E(e_0 \dots e_{k-1})$ determined by the initial segments $e=e_0\dots e_{k-1}$, and by $\alpha_k(x)$ the letter of $A_k$ corresponding to the cell of $\alpha_k$ to which a path $x \in X$ belongs.
Then the {\em $k$-coding} $\phi_k: X \to A_k^{\mathbb{N}}$ satisfies $(\phi_kx)_i=a$ if and only if $T^ix \in E(e)$, and $e$ corresponds to $a\in A_k$.
Then $\phi_k(x)= \alpha_k(x)\alpha_k(Tx) \dots$. The {\em $n$-symbols} of \cite{dm2008} provide a convenient way to picture simultaneously all $k$-codings for $k \leq n$ of orbit segments of length $\dim(n,j)$ of minimal segments to vertices $v(n,j)$ at level $n$. When the context is clear, we will refer to $s_k$ just as $s$.
By a {\em transitive (respectively forward transitive) path we mean one whose orbit (respectively forward orbit)} intersects every nonempty cylinder set.

Fix $k \geq 1$, $n \geq k$, and $j \in \{1,\dots,K_n+1\}$.
Let $x$ be a path in $X$ that follows only minimal edges from the root to $v(n,j)$.
Then $B^{(k)}(n,j)=\alpha_k(x) \alpha_k(Tx) \dots \alpha_k(T^{\dim(n,j)-1}x)$ is called the {\em basic block in the $k$-coding} at $v(n,j)$. So each basic block in the $k$-coding is an element of $A_k^*=\cup_{m > 0}A_k^m$, where $A_k^m$ is the set of all words of length $m$ in the alphabet $A_k$.
When $k=1$ we shall just refer to this as  the ``basic block" at $v(n,j)$ and denote it by $B(n,j)$.
\tb{We assume 
\begin{condition}\label{cond:cond3}
for each $n\geq k$ and each $j \leq K_n$ the minimal edge to $v(n,j)$ does not have \tb{ as its source the spacer vertex $v(n-1,K_{n-1}+1)$}, so that basic blocks $B^{(k)}(n,j), 1 \leq j \leq K_n$, in the $k$-coding do not begin with \tb{the spacer symbol $s_k$}.
\end{condition}}

Every segment of length $k$ is an extension of a segment of each length $k'<k$, and hence there is a (one-block)
map $\pi_{k,k'}$ that takes $B^{(k)}(n,j)$ to $B^{(k')}(n,j)$ by replacing the symbol associated to each segment of length $k$ by the symbol  associated to its initial segment of length $k'$.

There is a natural correspondence between adic systems as above and the cutting and stacking constructions familiar in ergodic theory.
Let $D_1, \dots ,D_{K_1}$ be finite blocks on the alphabet $A_1\setminus \{s_1\}$ such that every symbol of $A_1\setminus \{s_1\}$ appears in exactly one $D_i$.
(See below for why we use blocks rather than symbols.)
The basic blocks in the 1-coding at all levels $n > 1$ can be constructed by a recursive scheme, in which the parameters $g(n,j,i) \in \{1,\dots,K_n\}$, $a(n,j,i)\geq 0$, and $q(n,j) > 0$ are determined by the diagram, as follows:

\be
\begin{aligned}
B(1,j)&=D_j   \text{ for } j=1,\dots,K_1, \quad B(1,K_1+1)=s;\\
&\text{ and for } n > 0 \text{ and } j=1,\dots,K_{n+1},\\
B(n+1,j)&=B(n,g({n,j,0})) s^{a(n,j,0)} B(n,g({n,j,1})) s^{a(n,j,1)} \dots\\
&s^{a(n,j,q(n,j)-2)}B(n,g(n,j,q(n,j)-1)) s^{a(n,j,q(n,j)-1)};\\
  &\text{while }B(n+1,K_{n+1}+1)=s.
\end{aligned}\label{eq:recursion}
\en
See \cite[Definition 7.5.6]{FerencziMonteil2010} for a similar recursion, with an added recognizability condition.

If we code paths in the diagram by initial segments of length  $k > 1$, then the alphabet changes to $A_k$ and there are usually multiple segments from the root to each vertex at level $k$. The basic blocks in the $k$-coding at all levels $k$ and higher can be constructed by telescoping the adic system from the root vertex $v(0,0)$ to level $k$ and applying  the same recursive scheme.
Specifically, for any $n>k$ and $ j=1,\dots,K_n$, $B(n,j)$ in the new diagram corresponds to $B^{(k)}(n+k-1, j)$ in the original diagram. In (\ref{eq:recursion}) above, we are starting with blocks $D_j$ on $A_1 \setminus \{s_1\}$ instead of symbols in $A_1$
{because the diagram in question may have resulted
from telescoping from the root to level $k$ of a previous diagram that had single edges from the root to level $1$.}

Such a recursive scheme specifies how starting with $K_1+1$ intervals (the last one denoted by $s$ and thought of as a spacer) the resulting towers are to be cut and stacked in some order with strings of lengths $a(n,j,i)$ of spacers in between.
The blocks $B(n,j)$ will specify in the resulting measure-preserving system the itinerary of a point through the cells of the partition according to the $K_1+1$ intervals at the initial stage.
Conversely, such a recursive scheme determines the diagram to which it corresponds. See Figure \ref{fig:rec}.

\begin{figure}
  \begin{minipage}[t]{.45\textwidth}
  \centering
  \begin{tikzpicture}
    \foreach \i in {0,2,4}{
    \foreach \j in {-2,0,2,4}{
    \fill(\i,\j) circle (4pt);}}
    \foreach \i in{0,2}{
    \foreach \j in {-2,0,2}{
    \draw(\i,\j)--(0,\j+2);
    \draw(\i,\j)--(2,\j+2);
    \draw(\i,\j)--(4,\j+2);}}
    \fill(2,5) circle (4pt);
    \foreach \i in {0,4}{
    \draw(\i,4)--(2,5);}
    \draw[rounded corners](2,5)--(.8,4.8)--(0,4);
    \draw[rounded corners](2,5)--(1.2,4.3)--(0,4);
    \draw[rounded corners](2,5)--(1.8,4.5)--(2,4);
    \draw[rounded corners](2,5)--(2.2,4.5)--(2,4);
    \draw(4,4)--(4,-2);
    \draw[rounded corners](4,4)--(2,3.3)--(0,2);
    \draw[rounded corners](4,4)--(2,2.7)--(0,2);
    \draw[rounded corners](4,4)--(3,3.3)--(2,2);
    \draw[rounded corners](4,4)--(3,2.7)--(2,2);
    \node at (-.1,2.4){\tiny 1};
    \node at (.5,2.7){\tiny 4};
    \node at (.7,2.5)[circle,fill=white, inner sep=.5pt]{\tiny 2};
    \node at (1,2.5)[circle,fill=white,inner sep=0pt]{\tiny 3};
    \node at (.9,2.2)[circle,fill=white,inner sep=0pt]{\tiny 5};
    \node at (1.7,2.25)[circle,fill=white,inner sep=0pt]{\tiny 3};
    \node at (2,2.5)[circle,fill=white,inner sep=1pt]{\tiny 1};
    \node at (2.3,2.55)[circle,fill=white,inner sep=0pt]{\tiny 2};
    \node at (2.7,2.7)[circle,fill=white,inner sep=0pt]{\tiny 4};
    \node at (2.6,2.3)[circle,fill=white,inner sep=0pt]{\tiny 5};
  \end{tikzpicture}
  \subcaption{The above edge ordering at level 1 corresponds to a recursion beginning $B(2,1)=B(1,1)s^2B(1,2)s^1$ and $B(2,2)=B(1,2)sB(1,1)s^2$.}
  \label{fig:rec}\end{minipage}\hfill\begin{minipage}[t]{.45\textwidth}
  \begin{center}
   \begin{tikzpicture}
    \foreach \i in {0,2,4,6}
    {\fill(0,\i) circle (4pt);
    \fill (2,\i) circle (4pt);}
    \fill (1,7) circle (4pt);
    \draw (1,7)--(0,6)--(0,4);
    \draw (1,7)--(2,6)--(2,0);
    \draw (2,6)--(0,4);
    \foreach \i in {2,4}
    {\draw (2,\i)--(0,\i-2);
     \draw[rounded corners] (0,\i)--(.4,\i-1)--(0,\i-2);
     \draw[rounded corners] (0,\i)--(-.4,\i-1)--(0,\i-2);
     \node at (-.5,\i-1){\tiny 1};
     \node at (.5,\i-1){\tiny 3};
     \node at (1,\i-1.2)[circle,fill=white,inner sep=0pt]{\tiny 2};}
     \node at (0, 6)[rectangle,draw,fill=white]{0};
     \node at (2,6)[rectangle,draw,fill=white]{$s$};
  \end{tikzpicture}
\end{center}
  \subcaption{$B(3,1)$ appears both explicitly and non-explicitly in the coding of the minimal path through $v(4,1)$.}
\label{fig:explicit}
\end{minipage}
\caption{}
\end{figure}

Any diagram for which the far right vertex at level $n$ has the far right vertex at level $n-1$ as its only source can be ordered to correspond to some such recursion.
We make the standing assumption that (cf. \cite{Berthe2017}) the diagram and its associated recursion system are \tb{{\em growing}:
\begin{condition}\label{cond:cond4}
$\lim_{n \to \infty} \min\{ |B(n,j)|:{j \leq {K_n}}\}  = \infty$.
\end{condition}
}

 Given a vertex $v(n,j)$, a minimal segment from the root down to $v(n,j)$ defines a minimal cylinder set $C(n,j)$.
 For $1 \leq k \leq n < m, 1 \leq j \leq K_n +1, 1 \leq i \leq K_m +1$, and $x \in C(n,j)$, let
 \be
 \mathscr S (n,j;m,i)=\{p=0,1,\dots, |B^{(k)}(m,i)|-1 : T^px \in C(n,j)\}
 \en
 denote the set of positions in $B^{(k)}(m,i)$ at which $B^{(k)}(n,j)$ ``appears explicitly" in the recursive construction.
 More precisely, we define an {\em explicit appearance} of $B^{(k)}(n,j)$ in $B^{(k)}(m,i)$ to be a subblock of $B^{(k)}(m,i)$ which equals $B^{(k)}(n,j)$ and appears in $B^{(k)}(m,i)$ at position $p$ for some $p \in \mathscr S(n,j;m,i), 0 \leq p \leq |B^{(k)}(m,i)| - |B^{(k)}(n,j)|$.
 An {\em explicit appearance} of $B^{(k)}(n,j)$ in $\phi_k(y)$, for $y \in X$, is a subblock of $\phi_k(y)$ which equals $B^{(k)}(n,j)$ and appears in a position
 \be
 p \in \mathscr T(n,j;y) = \{p \in \mathbb N: T^p y \in C(n,j)\}.
 \en

For example consider the recursion given by:
\be \begin{array}{cl}
    B(1,1)= & 0 \\
    B(2,1)= & B(1,1)s \\
    B(3,1)=& B(2,1)sB(2,1) \\
    B(4,1)= & B(3,1)sB(3,1)
  \end{array}\en
Then we have that $B(3,1)=0ss0s$ and $B(4,1)=0ss0ss0ss0s$. A corresponding Bratteli Vershik system is pictured in Figure \ref{fig:explicit}. If $x$ is a minimal path passing through $v(4,1)$ we would say that $B(3,1)$ appears explicitly in $\phi(x)$ starting at position 0 and again at position 6, since $x$ and $T^6x$ are both in $C(3,1)$. However, $B(3,1)$ also appears in $\phi(x)$ starting at position 3. This is not an explicit appearance of $B(3,1)$ since $T^3x$ is not in $C(3,1)$.

 {It will also occasionally be helpful to think not of {coding paths} by initial segments of length $k$ but {\em by the vertices at level $k$}, which results from mapping each initial segment of length $k$ to its terminal vertex.}
  Given the coding by vertices at level $k$, we can expand each vertex to its basic block to get the $k$-coding, except without knowledge of \tb{the placement of} the initial coordinate.
 Conversely, given a $k$-coding and the set of positions at which each basic block appears explicitly, we can replace those basic blocks by the unique vertices to which they correspond to get the coding by vertices at level $k$.

The adic system in the example above corresponds to a {\em rank one} cutting and stacking construction. Each rank one cutting and stacking construction leads to the following constructive symbolic recursive system on the alphabet $\{0,s\}$, in which $B_n$ corresponds to $B(n,1)$:
\be
\begin{aligned}
	&B_0=0,\\
	&B_{n+1}=B_n s^{a(n,0)}B_n \dots B_n s^{a({n,q_n-1})} \text{ for } n \geq 0.
	\end{aligned}
\en
In this setup we are assuming there are only two edges leaving the root vertex.
However, when we code by initial segments of length $k$, we can refer back to Equation \ref{eq:recursion} and then with respect to the alphabet $A_k$ we would start with $B^{(k)}_0=D_1=01\dots d_k-2$.

For such a rank one system, if
\be\label{eq:sumcond}
 \sum_{n=0}^\infty \frac{\sum_{i=0}^{q_n-1}a(n,i)}{q_n|B_n|} < \infty,
 \en
then there is a unique non-atomic finite shift-invariant measure on the subshift consisting of all two-sided sequences in $\{0,s\}^\N$ all of whose finite subblocks appear as subblocks of the $B_n$ (see, for example \cite{Ferenczi1997,AFP}).

\section{Periodic codings of rank one systems}

\begin{theorem}\label{thm:rankone}
	Let $B_n, n \geq 0,$ be the sequence of basic blocks in a constructive symbolic rank one construction as above and let $\omega \in \{0,s\}^\N$ be the one-sided infinite sequence such that for each $n \geq 0, \omega=B_n \dots .$
	Suppose that $\omega=\omega_0\omega_1\dots$ is periodic and $P$ is the block of minimal length $|P| \geq 1$ such that $\omega=PPP\dots.$
	Denote by $x_{\min} \in X$ the infinite path from the root that for all $n$ follows the minimal edge entering $v(n,1)$.
	Then: \\
	(1) There are $N \in \N$ and $a \geq 0$ such that for all $n \geq N$ we have $a(n,q_n-1)=0$ and for all $i<q_n-1$ all $a(n,i)=a$.\\
	(2) For every $k \geq 1$ the $k$-coding of $x_{\min}$ by the first $k$ edges is periodic.\\
	(3) With its unique nonatomic invariant measure the system is measure-theoretically isomorphic to an odometer.\\
	(4) If $a=0$ the restriction of $T$ to $X\setminus \{x_s\}$  is topologically conjugate to an odometer.
\end{theorem}

\begin{proof}
(1) Recall that
two finite words on a finite alphabet $A$ commute if and only if they are powers of the same word: $u,v \in A^*=\cup_{n \geq 0}A^n, uv=vu$ implies there are $w  \in A^*, i,j \in \N$ such that $u=w^i,v=w^j$
\cite[Prop. 1.3.2, p. 8]{Lothaire1983}
(see also \cite[Prop. 1.3.5]{Lothaire1983}, \cite[Fine and Wilf, Th. 8.14, p. 272]{Lothaire2002}).
Thus $P$ can appear in $\omega$ only at multiples of $|P|$, since otherwise we would find a factorization $P=uv=vu$, leading to $P=w^r$ for some $r$ and some block $w$ of length shorter than $|P|$.

Now suppose that $N$ is large enough that $|B_N|>|P|$.
If $n \geq N$ and there are explicit appearances of the blocks $B_n s^j B_n$ and $B_n s^{j'} B_n$ in $B_{n+1}$ and $j' > j$,
then $|B_n|+j$ and $|B_n|+j'$ are both multiples of $|P|$, so
 $j' -j$ has to be a multiple of $|P|$.
Then $j'\geq |P|$, so in the subblock $s^{j'} P$ of $\omega$ we see at least two appearances of $P$.
This forces $P = s^k$ for some $k \leq j'$, which is impossible, since $P$ begins with $0$.
Therefore, there is an $a_n \geq 0$ such that consecutive explicit appearances of $B_n$ in the construction are separated by $s^{a_n}$.  Likewise, there exists an $a_{n+1} \geq 0$ such that consecutive explicit appearances of $B_{n+1}$ in the construction are separated by $s^{a_{n+1}}$.
Thus for all $n \geq N$,
\be
\begin{aligned}
B_{n+2}&=[B_{n+1}] s^{a_{n+1}} [B_{n+1}] \dots\\
	  &=[B_n \dots B_n s^{a(n,q_n-1)}] s^{a_{n+1}}[B_n \dots ] \dots \\
	  	  &=B_n \dots B_n s^{a_n} [B_n \dots] \dots ,
	  	  \end{aligned}
  	  \en
  	 so that
 $a(n,q_n-1)+a_{n+1} = a_n$.  So
either $a(n,q_n-1)=0$ or  $a_{n+1} < a_n$.
Since we can have $a_{n+1}<a_n$ for only finitely many $n$,
 it follows that for all $n$ sufficiently large $a(n,q_n-1)=0$ and $a_{n+1}=a_n$.

(2) We showed that for large enough $n$, $n \geq N$,
\be\label{eq:rank1per}
B_{n+1} =B_n s^a B_n s^a B_n \dots s^a B_n = (B_n s^a)^{m_n} B_n.
\en
If also $n \geq N > k$, considering the $k$-coding of $x_{\min}$ shows that the blocks $B_{n+1}^{(k)}$ satisfy the same recursion, (\ref{eq:rank1per}), with the added superscripts of $(k)$.
	(For $ 0 \leq j < \dim(n+1,1)$, the paths $T^j x_{\min}$ are the same in both cases, just assigned different symbols, so as we increase $n$ the blocks concatenate in the same way.)
	Thus
	\be
	B_{n+2}^{(k)}  =(B_{n+1}^{(k)} s^a)^{m_{n+1}} B_{n+1}^{(k)} =
	[(B_{n}^{(k)} s^a)^{m_n} B_{n}^{(k)} s^a]^{m_{n+1}} \,  [(B_n^{(k)} s^a)^{m_n} B_n^{(k)}],
		\en
		showing eventually that the $k$-coding of $x_{\min}$ is $[B_n^{(k)} s^a]^\infty$.

(3) Denote by $\mu$ the unique nonatomic invariant Borel probability measure on the rank one system $(X,T)$
(which is supported on $X \setminus \{x_s\}$).
If every $k$-coding is periodic, the measure-theoretic factors that they determine are permutations of finitely many points.
Since the partitions $\alpha_k$ of $X$ according to the first $k$ edges generate the full sigma-algebra of $X$, the inverse limit of these periodic systems, which is an odometer, is measure-theoretically isomorphic to the full system $(X,T, \mu)$.

(4) In this case ($a=0$), the system consists of the isolated path $x_s$ where $T$ is fixed together with
either an odometer or a finite periodic orbit. The former occurs if and only if there are infinitely many levels for which $|r^{-1}(v(n,1))|\geq 2$. By the assumption that the blocks are always growing there are an infinite number of paths in $X$ and therefore the case when there are only finitely many $n$ for which $|r^{-1}(v(n,1))|\geq 2$ {does not occur}.

\end{proof}


 \section{Periodic codings of general Bratteli-Vershik systems}

In this section we consider more general Bratteli-Vershik systems, still satisfying \tb{the conditions (C\ref{cond:cond1})--(C\ref{cond:cond4}) in Section \ref{sec:setup}.}
We begin with the following definition that will help us find conditions for $k$-codings to be periodic and for the systems to be isomorphic or topologically conjugate to odometers, generalizing results in \cite{FPS2017}.
As mentioned in the Introduction and discussed more fully below, this definition generalizes the idea of {\em uniformly ordered} in \cite{FPS2017} to a sort of incomplete periodicity.

\begin{definition}\label{def:semi}
Given $k >0$ and $n \geq k$, suppose there exist $c_n\geq 0$ and a shortest nonempty block $U_n\in A_k^{*}$ for which neither the first nor last symbol of $U_n$ is $s_k$ \tb{ and} such that for every $j=1,\dots,K_n$ there exist $t(n,j)>0$ and $0 \leq l(n,j) \leq c_n$ such that $B^{(k)}(n, j) = (U_ns_k^{c_n})^{t(n,j)}(U_n)s_k^{l(n,j)}$.
We then say that level $n$ is {\em  semi $k$-periodic}.
\end{definition}
\begin{remark}
An alternative way to think of this definition is that one could code vertices at level $n$ by the vertices at level $k$. Then the vertex {coding} would {include} a finite word, $V_n$, of vertices at level $k$, which when expanded to the corresponding $k$-coding, gives $U_n$. Then semi $k$-periodic is equivalent to the coding by vertices at level $k$ of $B(n,j)$ \tb{being $(V_nv(k,K_k+1)^{c_n})^{t(n,j)}(V_n)v(k,K_k+1)^{l(n,j)}$.}
\end{remark}
\begin{example}
For example, consider the recursion given by:
\be
\begin{aligned}
K_1=2 &\hskip .2in B(1,1)=0\\
      &\hskip .2in B(1,2)=1\\
K_2=3 &\hskip .2in B(2,1)=B(1,1)sB(1,2)\\
&\hskip .2inB(2,2)=B(1,1)sB(1,2)\\
&\hskip .2inB(2,3)=B(1,2)sB(1,1)\\
K_3=2 &\hskip .2in B(3,1)=B(2,1)sB(2,3)s^2B(2,2)sB(2,3)\\
&\hskip .2in       B(3,2)=B(2,2)sB(2,3)s^2B(2,1)sB(2,3)s
\end{aligned}
\en
We would then say that level 3 is semi 1-periodic with $U_3=0s1s1s0$ and $B(3,1)=(U_3s^2)U_3$ and $B(3,2)=(U_3s^2)U_3s$.  However, level 2 is not semi 1-periodic since $B(2,1)= 0s1$ and  $B(2,3)= 1s0$. In addition, level 3 is not semi 2-periodic.
\end{example}
\begin{remark}
	Note that because of the one-block map $\pi_{k,k'}: A_k^* \to A_{k'}^*$, if a level is semi $k$-periodic, then it is also semi $k'$-periodic for all $k' \leq k$.
	\end{remark}
	
Also note that if you are coding by the first $k$ edges and level $k+1$ is semi $k$-periodic, the condition becomes a more restrictive local condition.  Specifically:
\begin{proposition}\label{prop:decomp}
   If level $k+1$ is semi $k$-periodic, then the corresponding block $U_{k+1}$ uniquely decomposes into the blocks $B^{(k)}(k,j)$ for $j=1,\dots,K_{k}+1$, and level $k+1$ is pseudo-complete (see Definition \ref{def:pseudocomplete}).  \label{prop:Decomposition}
\end{proposition}

\begin{proof}
 When coding by the first $k$ edges, each $B^{(k)}(k,j)$ consists of $\dim(k,j)$ distinct letters, and if $j_1\neq j_2$, then $B^{(k)}(k,j_1)$ and $B^{(k)}(k,j_2)$ have no symbols in common.
Therefore, if level $k+1$ is semi $k$-periodic, the corresponding block $U_{k+1}$ uniquely decomposes into the blocks $B^{(k)}(k,j)$ for $j=1,\dots,K_{k}+1$.

{Since each basic block $B^{(k)}(k+1,j)$ contains the same block $U_{k+1}$ and each vertex $w \in \mathcal V_{k} \setminus \{s_{k}\}$ is the source of an edge to some vertex $v \in \mathcal V_{k+1}\setminus \{s_{k+1}\}$, it follows that for all $w \in \mathcal V_k \setminus \{s_k\}$ and all $ v' \in \mathcal V_{k+1}\setminus \{s_{k+1}\}$ there is an edge from $w$ to $v'$.}
Hence level $k+1$ is also pseudo-complete.
\end{proof}

In Example \ref{ex:semikperiodic}, level $k+1$ is semi $k$-periodic.
\begin{example}
\label{ex:semikperiodic}
Consider a system with $K_k=2$ and given by the following recursion.
\be
\begin{aligned}
B^{(k)}(k+1,1)&=B^{(k)}(k,2)s^3B^{(k)}(k,1)s^4B^{(k)}(k,2)s^3B^{(k)}(k,1)s^4B^{(k)}(k,2)s^3B^{(k)}(k,1)s^2\\
B^{(k)}(k+1,2)&=B^{(k)}(k,2)s^3B^{(k)}(k,1)s^4B^{(k)}(k,2)s^3B^{(k)}(k,1)s^2
\end{aligned}
\en
In the corresponding Bratteli diagram if $e$ and $e'$ are two distinct edges connecting level $k$ and $k+1$ with $r(e),r(e')\neq v(k+1,K_{k+1}+1)$ and $\xi(e)=\xi(e') \mod 9$, then $s(e)=s(e')$.
\end{example}
\begin{remark}
The semi $k$-periodic property at level $k+1$ is a generalization of \emph{uniformly ordered at level k+1} that appears in \cite{FPS2017}. The latter is equivalent to the existence of a block $W_{k+1} \in A_k^*$ such that for each $j=1,\dots,K_{k+1}$, $B^{(k)}(k+1,j) = (W_{k+1})^{t(k+1,j)}$ for some $t(k+1,j)>0$ (each basic block at level $k+1$ is periodic with the same repeated subblock).
Uniformly ordered at level $k+1$ was shown in \cite{FPS2017} to be sufficient for the $k$-coding of a transitive path to be periodic. So if infinitely many levels are uniformly ordered, then the resulting system is topologically conjugate to an odometer.
When $l(k+1,j)=c_{k+1}$ for all $j=1,\dots,K_{k+1}$, the semi $k$-periodic property at level $k+1$ is equivalent to uniformly ordered at level $k+1$ provided that the basic block $B^{(k)}(k+1,K_{k+1}+1)=s_k$ corresponding to vertex $v(k+1,K_{k+1}+1)$ is disregarded. Otherwise, it is slightly weaker than uniformly ordered in that $W_{k+1} = U_{k+1}s^{c_{k+1}}$, and  $B^{(k)}(k+1,j)$ is allowed to stop short before completely running through the last set of spacers in $W_{k+1}$, leaving a deficit.
We will see in Theorem \ref{thm:general} and its corollary that semi $k$-periodicity is a necessary condition for the $k$-coding of a transitive path to be periodic.
\end{remark}

The results of this section assume the existence of a transitive path. The following lemma gives a sufficient condition for existence of a forward transitive path.
\begin{lemma}\label{lemma:transpath}
   Assume that the diagram and associated recursive symbolic construction satisfy \tb{the conditions (C\ref{cond:cond1})--(C\ref{cond:cond4})}
   specified in Section \ref{sec:setup} and that there are infinitely many $n \in \N$ such that level $(n+1)$ is semi $n$-periodic. Further, suppose that there are infinitely many $n$ for which $|s^{-1}v(n,K_n+1)|>1$. Then $X$ has a minimal forward transitive path.
\end{lemma}
\begin{proof}
Since it is assumed that no basic blocks begin with spacers, it is possible to construct a sequence of paths $(x(m))$ such that for each $m\in \mathbb{N}$, $x(m)$ is minimal into level $m$ and does not pass through $v(n,K_n+1)$ for any $n$ in $\mathbb{N}$. There is a convergent subsequence $(x(m_k))$ which converges to a path $x$. Then $x$ is minimal and does not pass through $v(n,K_n+1)$ for any $n\in \mathbb{N}$.

 When level $n+1$ is semi $n$-periodic it is pseudo-complete, and hence for infinitely many $n$ and any $1\leq i \leq K_n,1\leq j \leq K_{n+1}$, there is an edge from $v(n,i)$ to $v(n+1,j)$. For each $n$ denote by $v_n$ the vertex through which $x$ passes at level $n$.
  Let $y$ be an arbitrary path in $X$ and $N \in \mathbb N$.
  Choose $m > N$ such that level $m+1$ is semi $m$-periodic and hence pseudo-complete.
  Let $y_m$ be the $m$'th edge in the path $y$, with $r(y_m)=v(m,j_y)$.

If $j_y\neq K_{m}+1$, then by pseudo-completeness there is an edge from $v(m,j_y)$ to $v_{m+1}$.
  Since $x$ is minimal into $v_{m+1}$ there is a path in the forward orbit of $x$ which agrees with $y$ down to level $m$ and hence to level $N$.

  If $j_y=K_{m}+1$, there exists an $N_2\geq m$ such that there is an edge from $v(N_2,K_{N_2}+1)$ to a vertex $v(N_2+1,j')$ where $j'\neq K_{N_2+1}+1$.
   Find $m_2 > N_2$ such that level $m_2+1$ is semi $m_2$-periodic, hence pseudo-complete.
   There is a segment from $v(N_2+1,j')$ to some vertex $v(m_2,r), 1\leq r \leq K_{m_2}$,
   and an edge from $v(m_2,r)$ to $v_{m_2+1}$.
   Then $x$ passes through $v_{m_2+1}$ and there is a segment connecting $v_{m_2+1}$ to $v(N_2,K_{N_2}+1)$ and hence to $v(N,K_{N}+1)$. Since $x$ is minimal into $v_{m_2+1}$, there is a path in the forward orbit of $x$ which agrees with $y$ down to level $N$.
   Hence $x$ is a forward transitive path for the system.
\end{proof}

\begin{remark} \tb{Assume the hypotheses of Lemma \ref{lemma:transpath}.} If we restrict to systems for which $X'=X\setminus x_s$ has a unique minimal path, then this path is transitive. Conversely, if $X$ has a transitive path, then there are necessarily infinitely many $n$ for which $|s^{-1}(v(n,K_n+1))|>1$.   Indeed, if there are only finitely many $n$ for which $|s^{-1}v(n,K_n+1)|>1$, then $x_s\in X$ is an isolated fixed path. However, we could exclude $x_s$ and the above argument would still apply to $X'=X \setminus \{x_s\}$ to show that $X'$ has a transitive path.
\end{remark}

 The following definition adds a condition to level $n$ being semi $k$-periodic.
   Recall that level $n >k$ is semi $k$-periodic if for each \tb{$j=1,\dots, K_n$} we have $B^{(k)}(n, j) = (U_ns^{c_n})^{t(n,j)}(U_n)s^{l(n,j)}$ for some finite block $U_n\in A_k^{*}$ and $l(n,j) \leq c_n$.

   \begin{definition}\label{def:locdeficit}
  We say
    that {\em level $n$ satisfies the local deficit condition with respect to level $k$}, abbreviated $LDC(n,k)$, if
     level $n$ is semi $k$-periodic and 
     for each $i=1, \dots, K_{n+1}$ and $j,j''=1,\dots,K_{n}$
    and explicit appearances of $B^{(k)}(n,j)$ and $B^{(k)}(n,j'')$ in any recursion
    \be
    B^{(k)}(n+1,i)=\dots B^{(k)}(n, j)s^{a(n,j)}B^{(k)}(n,j')\dots B^{(k)}(n,j'')s^{a}
    \en
     we have that the deficit $a(n,j)=c_n-l(n,j)$ and $a\leq a(n,j'')= c_n-l(n,j'')$.
   \end{definition}

\begin{remark}\label{rem:nottoomanyspacers}
If there exists an $N >k$ such that for all $n \geq N$
$LDC(n,k)$ is satisfied, then for all these $n$ we have semi $k$-periodicity, in other words the basic blocks at level $n$ are periodic with a deficit of spacers at the end.
Furthermore, before a basic block $B^{(k)}(n, j)$ at level $n$ is concatenated with another in the construction of a basic block at level $n+1$,  just enough spacers are added at the end of $B^{(k)}(n, j)$ to complete the period.  We shall show in Theorem \ref{thm:general} that this implies that the $k$-coding of the minimal orbit is periodic.
\end{remark}

\begin{remark}
   The condition $LDC(n,k)$ can be seen on an ordered Bratteli diagram in the following manner. Let $k<n$ and suppose $LDC(n,k)$ is satisfied.
   Let $v(n,j)$ and $v(n+1,j')$ be two vertices such that there is an edge $e$ connecting them. Further assume that there exists another edge $e'$ with $r(e')=v(n+1,j')$, $s(e')\neq K_{n}+1$, and $\xi(e)<\xi(e')$.
   In other words, $B(n,j)$ is not the last basic block appearing explicitly in the decomposition of $B(n+1,j')$. 
   Since level $n$ is semi $k$-periodic, \tb{$B^{(k)}(n,j)=(U_ns^{c_n})^{t(n,j)}(U_n)s^{l(n,j)}$.} For ease of notation, let $m= c_n -l(n,j)$.
   Then there are $m+1$ edges, $e_1, e_2,\dots, e_{m+1}$, into vertex $v(n+1,j')$ for which $\xi(e_i)=\xi(e)+i$, $s(e_1)=s(e_2)=\dots s(e_m)=K_n+1$ and $s(e_{m+1})\neq K_n+1$; equivalently, in the partial edge ordering $e$ is followed by exactly $c_n-l(n,j)$ edges connecting $v(n+1,j')$ to the spacer vertex.
   Since the partial edge ordering is fixed, we know that if $LDC(n,k)$ is satisfied for multiple values of $k$, the deficit is the same for each $k$.
\end{remark}

Proposition \ref{prop:LocalDeficit} equates satisfying $LDC(n,k)$ to semi $k$-periodicity at both levels $n$ and $n+1$ plus a condition on the corresponding blocks $U_n$, $U_{n+1}$ and the numbers of consecutive spacers at those levels. It is beneficial to have both formulations in the proofs of the following theorems.

\begin{proposition}
 Given $n>k$, if level $n$ and level $n+1$ are both semi $k$-periodic with $U_{n+1}=U_n$, then $c_{n+1} = c_n$ and $LDC(n,k)$ is satisfied. Conversely, if $LDC(n,k)$ is satisfied, then levels $n$ and $n+1$ are semi $k$-periodic with $U_n=U_{n+1}$ and $c_n=c_{n+1}$.\label{prop:LocalDeficit}
\end{proposition}
\begin{proof}
First, we assume that level $n$ and level $n+1$ are both semi $k$-periodic with $U_{n+1}=U_n$. Then
for every $j=1, \dots, K_n$,
\be
B^{(k)}(n,j)=(U_ns^{c_n})^{t(n,j)}U_ns^{l(n,j)},
\en
where $U_n$, $c_n$, $t(n,j)$, and $l(n,j)$ are as in Definition \ref{def:semi}.
Furthermore, for each $1\leq j'\leq K_{n+1}$ there are $m(j')\in \N$, $j_1,\dots , j_{m(j')}\leq K_n$ and $a_1,\dots , a_{m(j')}\geq 0$ such that
\be
\begin{aligned}
&B^{(k)}(n+1,j')=B^{(k)}(n,j_1)s^{a_1}B^{(k)}(n,j_2)s^{a_2}\dots B^{(k)}(n,j_{m(j')})s^{a_{m(j')}}\\
&=(U_ns^{c_n})^{t(n,j_1)}U_ns^{l(n,j_1)}s^{a_1}(U_ns^{c_n})^{t(n,j_2)}U_ns^{l(n,j_2)}s^{a_2}\dots  (U_ns^{c_n})^{t(n,j_{m(j')})}U_ns^{l(n,j_{m(j')})}s^{a_{m(j')}}.
\end{aligned}
\en
Also,
\be
B^{(k)}(n+1,j')=(U_{n+1}s^{c_{n+1}})^{t(n+1,j')}U_{n+1}s^{l(n+1,j')}.
\en
Since we also have $U_{n+1}=U_n$, and $U_n$ does not begin or end with $s$, we must have $c_{n+1}=c_n$.
Then for all $1\leq i<m(j')$, we have $l(n,j_i)+a_i=c_n$.
Furthermore, $l(n,j_{(m(j')})+a_{m(j')} = l(n+1,j') \leq c_{n+1}=c_n$ implies that $a_{m(j')}\leq c_n-l(n,j_{m(j')})$, as required. Hence $LDC(n,k)$ is satisfied.

Now, we assume that $LDC(n,k)$ is satisfied. Then level $n$ is semi $k$-periodic, and for $1\leq j\leq K_n$ we have
\be
B^{(k)}(n,j)=(U_ns^{c_n})^{t(n,j)}U_ns^{l(n,j)}.
\en
Let $j' \in \{1,\dots,K_{n+1}\}$. Then
\be
\begin{aligned}
B^{(k)}&(n+1,j')=B^{(k)}(n,j_1)s^{c_n-l(n,j_1)}B^{(k)}(n,j_2)s^{c_n-l(n,j_2)}\dots B^{(k)}(n,j_{m(j')})s^{a_{m(j')}}\\
&=(U_ns^{c_n})^{t(n,j_1)}U_ns^{c_n}(U_ns^{c_n})^{t(n,j_2)}U_ns^{c_n}\dots (U_ns^{c_n})^{t(n,j_{m(j')})}U_ns^{l(n,j_{m(j')})}s^{a_{m(j')}}\\
&=(U_ns^{c_n})^{m(j')-1+\sum_{i=1}^{m(j')}t(n,j_i)}U_ns^{l(n,j_{m(j')})+a_{m(j')}}.
\end{aligned}
\en
In order to fulfill the requirements of Definition \ref{def:semi}, we let $U_{n+1} = U_n$ and $c_{n+1}=c_n$, \tb{noting that we {\em must} have $c_{n+1}=c_n$ and no shorter block than $U_n$ can serve as $U_{n+1}$.}
{Then}
\be
B^{(k)}(n+1,j')=(U_{n+1}s^{c_{n+1}})^{t(n+1,j')}U_{n+1}s^{l(n+1,j')},
\en
where
\be
\begin{gathered}
 t(n+1,j')={m(j')-1+\sum_{i=1}^{m(j')}t(n,j_i)}, \\
 \text{and } l(n+1,j')=l(n,j_{m(j')})+a_{m(j')}.
 \end{gathered}
 \en
 Since $a_{m(j')}\leq c_n-l(n,j_{m(j')})$, we also have that $l(n+1,j')\leq c_n=c_{n+1}$. Then, since $j'$ was arbitrary, level $n+1$ is semi $k$-periodic, $U_{n+1}=U_n$ and $c_{n+1} = c_n$.
\end{proof}

\begin{remark}\label{rem:lotsofspacers}
Note that if there exists  $N>k$ such that $LDC(n,k)$ is satisfied at all levels $n \geq N$, then we have semi $k$-periodicity at each level $n \geq N$ and, by Proposition \ref{prop:LocalDeficit}, $c_n$ = $c_N$ and $U_n = U_N$.
Hence, for any basic block
\be
B^{(k)}(n+1,i)=\dots B^{(k)}(n, j)s^{a(n,j)}B^{(k)}(n,j')\dots B^{(k)}(n,j'')s^{a}, i \leq K_{n+1},
\en
  the total number of spacers $a + l(n,j'')$ that appear at the end of this block (i.e. following the last non spacer) cannot exceed $c_n = c_N$. So
  {if $LDC(n,k)$ is satisfied for all large enough $n$,} we can only add spacers at the ends of basic blocks for finitely many levels and therefore arbitrarily long strings of spacers do not appear in the $k$-coding.
\end{remark}

We are now ready to prove that if there is a $k$ such that the $k$-coding of a minimal forward transitive path is periodic then all basic blocks of sufficient length must share this periodic structure. More precisely, from some point on, all levels must satisfy the local deficit condition with respect to level $k$ (which includes that each level is semi $k$-periodic).
	By telescoping from the root to level $k$, we may assume that $k=1$. The plan of the proof is to note first that the uniqueness of the minimal path implies that eventually all basic blocks must begin with the fundamental repeating block $P$.
	Since blocks need to concatenate in such a way as to preserve periodicity, and the fundamental repeating block $P$ cannot overlap itself, all basic blocks must either end with a complete $P$ or else be able to make up the end of $P$ at the next level.
The only way this can be accomplished is by adding the symbol $s$ exactly the right number of times, showing that from some level on, $LDC(n,1)$ is satisfied.
We show conversely that if $LDC(n,1)$ is satisfied for all sufficiently large $n$, then each {long enough} basic block must be periodic up to a deficit. Therefore any minimal path will have a 1-coding that begins with an arbitrarily long string of $P$ and hence is periodic.

\begin{theorem}\label{thm:general}
Assume that the diagram and associated recursive symbolic construction satisfy \tb{the conditions (C\ref{cond:cond1})--(C\ref{cond:cond4})} specified in Section 2 and there is a unique minimal path $x\in X'=X \setminus\{x_s\}$ that is forward transitive. Then the $1$-coding  $\omega = \phi_1(x)$ is periodic if and only if there is an $N$ such that every level $n \geq N$ satisfies the local deficit condition with respect to level $1$ ($LDC(n,1)$ holds). \end{theorem}

\begin{proof} Assume that the $1$-coding $\omega=\phi_1(x)$  of the minimal transitive path $x \in X$ is  periodic with minimal repeating block $P$.
As in the proof of Theorem \ref{thm:rankone}, $P$ appears only at multiples of $|P|$.
Because we are assuming that the procedure is growing, we may choose $N$ large enough to ensure that for $n\geq N$, $|B(n, j)| \geq 2|P|$ for all $j \leq K_n$.

We claim that once $n$ is large enough that all basic blocks $B(n,j), 1\leq j \leq K_n$, have length greater than $|P|$, they all begin with $P$.
For otherwise we can find infinitely many vertices $v(n_m,i_{n_m}), 1 \leq i_{n_m} \leq K_{n_m}$, for which $B(n_m,i_{n_m})$ does not begin with $P$ (and recall that they cannot begin with $s$).
For each $m$ let $x_m$ be an infinite path that is minimal from the root to $v(n_m,i_{n_m})$, so that the $1$-coding of $x_m$ begins with $B(n_m,i_{n_m})$. Let $x'$ be the limit of a convergent subsequence of the paths $x_m$.
Then $x'$ is a minimal path and
 each initial block of its $1$-coding is an initial block of the $1$-coding of some $x_m$, so that
 its $1$-coding cannot begin with $P$.
 Thus $x' \neq x$, contradicting our assumption about uniqueness of the minimal path. Hence we have shown that for sufficiently large $n$, all basic blocks begin with $P$.

Suppose that $G_1=B(n,i)s^mB(n,j)$ and $G_2=B(n,i)s^{m'}B(n,j')$ are explicit appearances in $\omega$ for some $i , j, j'  \leq K_{n}$ and $0\leq m<m'$.  The $P$'s cannot overlap; $B(n,i)$, $B(n,j)$, and $B(n,j')$ all begin with $P$; and $\omega=PPPP\dots$.
Consider the last place in the initial subblock $B(n,i)$ of $G_1$ at which an appearance of $P$ begins.
	The next place in $G_1$ which initiates an appearance of $P$ must be the first place in $B(n,j)$.
	Similarly, the next place in $G_2$ which initiates an appearance of $P$ must be the first place in $B(n,j')$.
	Therefore $m=m'$. (This is essentially the same argument as in the proof of Theorem \ref{thm:rankone}.)
So for every $n \geq N$ and $i \leq K_n$, there exists $m(n,i) \geq 0$ such that each explicit appearance of  $B(n,i)$  in $\omega$ is followed by $s^{m(n,i)}P$.

Choose $i \leq K_n$ such that for all $j \leq K_n$ we must have $|B(n, j) s^{m(n,j)}| \leq |B(n, i) s^{m(n,i)}|$.
  Then $B(n,j) s^{m(n,j)}  = P^t$ for some $t>0$ and $B(n, i)s^{m(n,i)} = P^{r+t}$ for some $r\geq0$ ($t$ and $r$ depend on $n,j,i$).
  Choose $c$ as large as possible so that  $P=Us^{c}$ for some block $U \in A_1^{*}$ (noting that $c \geq m(n,j)$ for all $j$).
   It follows that $B(n,i) = (Us^{c})^{r+t-1}(U)s^{c-m(n,i)}$ and $B(n,j) = (Us^{c})^{t-1}(U)s^{{c}-m(n,j)}$. Hence level $n$ is semi $k$-periodic with $U_n =U$. Then by Proposition \ref{prop:LocalDeficit}, every level $n>N$ satisfies $LDC(n,1)$.

   Conversely, assume that there is an $N\in \mathbb{N}$ such that every level $n \geq N$ satisfies $LDC(n,1)$.
    Then for every $n\geq N$,  \tb{$j,j''=1,\dots, K_n$},  and each pair of explicit appearances of $B(n,j)$ and $B(n,j'')$ in any recursion
   \be
B(n+1,i)=B(n,j)s^{a(n,j)}B(n,j')\dots B(n,j'')s^a
\en
   we have that the deficit $c_n-l(n,j)=a(n,j)$ and $a\leq c_n-l(n,j'')$. Now, define $\widetilde{B}(n,j)$ such that $B(n,j)=\widetilde{B}(n,j)s^{l(n,j)}$. Expanding an arbitrary block at level $n+1$ in terms of the blocks at level $n$, we have
   \begin{align}
       B(n+1,i)& = \widetilde{B}(n,j)s^{c_n}\widetilde{B}(n,j')s^{c_n}\dots
   \end{align}
     Since every $n \geq N$ satisfies $LDC(n,1)$, by Proposition \ref{prop:LocalDeficit} we have that $c_n=c_{n-1}=\dots=c_N$. Hence,
     \begin{align}\label{eq:converse1}
         B(n+1,i) & =
    \widetilde{B}(n,j)s^{c_N}\widetilde{B}(n,j')s^{c_N}\dots.
     \end{align}
 Repeat these observations with $n+1,n$ replaced by $n, n-1$ and so on to arrive (for appropriate $d,d',\dots$) at
        \begin{align}
       B(n+1,i) &  =
    \widetilde{B}(N,d)s^{c_N}\widetilde{B}(N,d')s^{c_N}\dots
    \end{align}
     Further, since level $N$ is semi 1-periodic, \tb{we have that}
      \be
      \begin{aligned}
      \widetilde{B}(N,d) s^{c_N}&= \widetilde{B}(N,d)s^{l(N,d)}s^{a(N,d)}=B(N,d)s^{a(N,d)}\\
      	&=(U_Ns^{c_N})^{t(N,d)}U_Ns^{l(N,d)}s^{a(N,d)}=(U_Ns^{c_N})^{t(N,d)}U_Ns^{c_N}.
      	\end{aligned}
      \en

     Since the blocks continue to grow, by taking $n$ sufficiently large one can obtain an arbitrarily long string $(U_Ns^{c_N})^m$. Then since this is true for every block at level $n$, it is true for the block through which the unique minimal and transitive path passes at level $n$.
     We have shown that the $1$-coding of the unique minimal and transitive path in $X$ begins with an arbitrarily long string of concatenations of $(U_Ns^{c_N})$, so the 1-coding is necessarily periodic.
  \end{proof}

\begin{remark}\label{rem:MultipleMinPaths}
	
		If there is a transitive path with a periodic 1-coding, then every path has a periodic 1-coding, and the codings of the transitive paths are all (possibly truncated) shifts of one another. 
		\tb{Figure \ref{fig:TwoMinPaths} presents} a system which has two minimal transitive paths with periodic 1-codings that are shifts of each other.
		Theorem \ref{thm:general} can be generalized to handle diagrams with more than one transitive minimal path by extending the definitions of $k$-periodic (Definition \ref{def:semi}) and $LDC(n,k)$ (Definition \ref{def:locdeficit}) appropriately, focusing on sets of minimal paths that have the same codings and the same sets of basic blocks assigned to their vertices.
	\end{remark}

		\begin{figure}
		\begin{minipage}[t]{.45\textwidth}
			\begin{center}
				 \begin{tikzpicture}
				\foreach \i in {0,2,4}{
					\foreach \j in {0,2,4,6}{
						\fill(\i,\j) circle (4pt);}}
				\fill(2,7) circle (4pt);
				\draw(2,7) -- (0,6);
				\draw(2,7) --(2,6);
				\draw(2,7)--(4,6);
				\draw[rounded corners](0,6)--(-.5,5)--(0,4);
				\draw[rounded corners](0,6)--(.5,5)--(0,4);
				\draw(0,6)--(2,4);
				\draw(2,6)--(0,4);
				\draw[rounded corners](2,6)--(1.5,5)--(2,4);
				\draw[rounded corners](2,6)--(2.5,5)--(2,4);
				\draw[rounded corners](4,6)--(2,5.25)--(0,4);
				\draw[rounded corners](4,6)--(2,4.75)--(0,4);
				\draw[rounded corners](4,6)--(3,5.25)--(2,4);
				\draw[rounded corners](4,6)--(3,4.75)--(2,4);
				\draw(4,6)--(4,0);
				\foreach \i in {2,4}{
					\draw[rounded corners](0,\i)--(-.5,\i-1)--(0,\i-2);
					\draw[rounded corners](0,\i)--(.5,\i-1)--(0,\i-2);
					\draw(0,\i)--(2,\i-2);
					\draw(2,\i)--(0,\i-2);
					\draw[rounded corners](2,\i)--(1.5,\i-1)--(2,\i-2);
					\draw[rounded corners](2,\i)--(2.5,\i-1)--(2,\i-2);
					\draw(4,\i)--(0,\i-2);
					\draw(4,\i)--(2,\i-2);
					\node at (-.4,\i-1.4){\tiny{1}};
					\node at (.45,\i-1)[fill=white]{\tiny{4}};
					\node at (.6,\i-1.3)[fill=white]{\tiny{2}};
					\node at (.5,\i-1.9){\tiny{3}};
					\node at (1.4,\i-1.6){\tiny{3}};
					\node at (1.8,\i-1.4){\tiny{1}};
					\node at (2.2,\i-1.4){\tiny{4}};
					\node at (2.6,\i-1.6){\tiny{2}};
					
				}
				\draw(4,6)--(2,4);
				\node at (0,4) [rectangle,draw, fill=white]{$01s^20$};
				\node at (2,4)[rectangle,draw,fill=white]{$1s^201s$};
				\node at (0,6)[rectangle,draw,fill=white]{$0$};
				\node at (2,6)[rectangle,draw,fill=white]{$1$};
				\node at (4,6)[rectangle,draw,fill=white]{$s$};
				\node at (4,4)[rectangle,draw,fill=white]{$s$};
				
				\end{tikzpicture}
				\subcaption{A system with two minimal forward transitive paths. The diagram is stationary after level 2.}\label{fig:TwoMinPaths}
			\end{center}
			\end{minipage}\hfill\begin{minipage}[t]{.45\textwidth}
			\begin{center}
\begin{tikzpicture}
\foreach \i in {0,2,4}{
\draw(2,7)--(\i,6);
\foreach \j in {0,2,4,6}{
 \fill(\i,\j) circle (4pt);}}
 \fill(2,7) circle (4pt);
 \foreach \i in {0,2}{
 \foreach\j in {0,2,4}{
 \draw(\i,\j)--(0,\j+2);
 \draw(\i,\j)--(2,\j+2);
 \draw(\i,\j)--(4,\j+2);}}
 \draw(4,6)--(4,0);
 \foreach \i in {0,2}{
 \node[circle,fill=white, inner sep=0pt]at (-.2,\i+.4){\tiny 1};
 \node[circle,fill=white, inner sep=0pt] at (.4,\i+.6){\tiny 3};
 \node[circle,fill=white, inner sep=1pt]at (0.6,\i+.1){\tiny 2};
 \node[circle,fill=white, inner sep=0pt]at (1.6,\i+.15){\tiny 3};
 \node[circle,fill=white, inner sep=1pt]at (2,\i+.4){\tiny 1};
 \node[circle,fill=white, inner sep=1pt]at (2.4,\i+.15){\tiny 2};
}
 \node[circle,fill=white, inner sep=0pt]at (-.2,4.4){\tiny 1};
 \node[circle,fill=white, inner sep=0pt]at (.3,4.5){\tiny 2};
 \node[circle,fill=white, inner sep=0pt]at (0.5,4.1){\tiny 3};
  \node[circle,fill=white, inner sep=0pt]at (1.6,4.25){\tiny 1};
 \node[circle,fill=white, inner sep=1pt] at (2,4.4){\tiny 2};
 \node[circle,fill=white, inner sep=1pt]at (2.45,4.2){\tiny 3};
\end{tikzpicture}
\end{center}
\subcaption{The 1-coding is periodic, but the 2-coding is not.}
	\label{fig:notper}
\end{minipage}
\caption{}
		\end{figure}

\begin{example}\label{ex:someper}
Unlike in the rank one case, having a transitive path for which the $k$-coding is periodic for a particular $k$ does not imply that every $k$-coding is periodic. Consider the following example. 
\tb{Figure \ref{fig:notper} shows part of the associated Bratteli-Vershik system, which is stationary after level 2.}
 Then the $1$-coding is periodic with least period $P_1=01ss$. However, the $2$-coding is not periodic.
 This example reappears as Example \ref{ex:someiso}.

	\be
\begin{aligned}
  B(2,1) & =B(1,1)B(1,2)s \\
  B(2,2) & =B(1,1)B(1,2)s\\
\end{aligned}
\en
And for $n\geq 3$,
\be
\begin{aligned}
  B(n,1) & =B(n-1,1)sB(n-1,2) \\
  B(n,2) & =B(n-1,2)sB(n-1,1)
\end{aligned}
\en
\end{example}

\begin{corollary}
  Every $k$-coding is periodic if and only if for every $k$ there is $N>k$ such that every $n \geq N$ satisfies the local deficit condition with respect to level $k$ ($LDC(n,k)$).
\end{corollary}

\begin{proof}
Telescope between the root and level $k$ and apply Theorem \ref{thm:general}.
\end{proof}

\begin{remark}
		For a diagram as in Theorem \ref{thm:general}, there exists a telescoping such that every $k$-coding is periodic if and only if in the original diagram every $k$-coding is periodic.
		(This is because the number of paths from the root to level $k$ and the partial ordering of those paths are preserved under telescoping, so the coding itself remains completely unchanged, and hence whether the local deficit condition, which includes semi $k$-periodicity, is satisfied or not also remains unchanged.)
	\end{remark}

As stated in the Introduction, if every $k$-coding is periodic, then with respect to any invariant ergodic measure the system is isomorphic to an odometer. If in addition eventually the spacer path stops branching, we will have the system {with the spacer path removed} topologically conjugate to an odometer.

\begin{proposition}\label{proposition:2}
 Assume that the diagram and associated recursive symbolic construction satisfy \tb{the conditions (C\ref{cond:cond1})--(C\ref{cond:cond4})} specified in Section \ref{sec:setup}.
 If the ordered Bratteli diagram can be telescoped so that every level satisfies the local deficit condition and there exists an $N$ such that for all $n\geq N$ $|s^{-1}v(n,K_n+1)|=1$, then the resulting system is topologically conjugate to an odometer together with one {isolated} fixed path.
\end{proposition}
\begin{proof}
For all $n\geq N$ there are no edges between $v(n,K_n+1)$ and $v(n+1,j)$ for any $j=1,\dots K_{n+1}$. Therefore, if we telescope to level $N+1$, the new diagram has the property that for all $n$, $v(n,K_n+1)$ has only one source, and as such the diagram decomposes into the isolated path and the rest of the diagram. Then if we exclude the isolated path, the remaining diagram is uniformly ordered according to the definition in \cite{FPS2017}.
\end{proof}

\section{Examples}\label{sec:examples}
In this section we systematically consider examples of various Bratteli-Vershik systems in regard to their level-$k$-factors: where all, some, or none of the $k$-factors are finite, and where all, some, or none of the $k$-factors are isomorphic to the original Bratteli-Vershik system with a fully supported ergodic invariant measure.
In particular, Example \ref{ex:someiso} shows that, unlike in the rank one case, it is possible to have some but not all $k$-factors finite, and some but not all isomorphic to the full system.
In the figures, the diagrams are assumed to continue in a stationary manner.

\begin{example}
  \label{ex:AllFiniteAndIso}
  We first consider the case in which every level-$k$ factor is finite and each $k$-factor is isomorphic to the original system. This implies that the original system contains finitely many points. In order for the original system to be isomorphic to its level-$1$ factor, the Bratteli-Vershik presentation must have the same number of edges from the root to level 1 as the number of points in the system.
  Then for any $n\geq 1$, and any $v(n,j)$, $|s^{-1}(v(n,j))|=1$. See Figure \ref{fig:AllFiniteAndIso}. This class of systems is not considered in our previous theorems, but its members are simple to examine separately.
  \end{example}

\begin{figure}
\begin{minipage}[t]{.45\textwidth}
  \centering
  \begin{tikzpicture}
    \foreach \i in {0,2,4}{
    \foreach \j in {0,2,4}{
    \fill(\i,\j) circle (4pt);}}
    \fill(2,5) circle (4pt);
    \foreach \i in {0,4}{
    \draw(2,5)--(\i,4);}
    \draw[rounded corners](2,5)--(.7,4.6)--(0,4);
    \draw[rounded corners](2,5)--(1.2,4.3)--(0,4);
    \draw[rounded corners](2,5)--(1.8,4.5)--(2,4);
    \draw[rounded corners](2,5)--(2.2,4.5)--(2,4);
    \foreach\i in {0,2,4}{
    \foreach \j in {2,4}{
    \draw(\i,\j)--(\i,\j-2);}}
  \end{tikzpicture}
  \subcaption{All $k$-factors are finite and isomorphic to the original system.}\label{fig:AllFiniteAndIso}
\end{minipage}\hfill  \begin{minipage}[t]{.45\textwidth}
  \centering
  \begin{tikzpicture}
   \foreach \i in {0,2}{
    \foreach \j in {0,2,4}{
    \fill(\i,\j) circle (4pt);}}
    \fill(1,5) circle (4pt);
    \draw(1,5)--(0,4)--(0,0);
    \draw(1,5)--(2,4)--(2,0);
    \foreach \i in {2,4}{
    \draw[rounded corners](0,\i)--(-.3,-1+\i)--(0,-2+\i);
    \draw[rounded corners](0,\i)--(.3,-1+\i)--(0,-2+\i);
    \draw(2,\i)--(0,-2+\i);
    \node at (-.3,-1.4+\i) {\tiny 1};
    \node at (0,-1.25+\i) [circle,fill=white, inner sep=1pt]{\tiny 2};
    \node at (.3,-1.3+\i) {\tiny 4};
    \node at (.75,-1.5+\i) {\tiny 3};
    }
  \end{tikzpicture}
  \subcaption{The Chacon system as an adic, where the numbers specify the edge orderings.}\label{fig:Chacon}
  \end{minipage}
  \caption{}
\end{figure}

\begin{example}
  \label{ex:AllPeriodicNotIsomorphic}
  Our next example is precisely the focus of much of the paper: a system for which every $k$-factor is finite, and none are isomorphic to the original system. In this case, given a finite ergodic measure, the system is isomorphic to an odometer. We give the recursion.
  \be \begin{aligned}
K_1=2 &\hskip .1in B(1,1)=D_1\\
 &\hskip .1in B(1,2)=D_2\\
\\
K_2=2 &\hskip .2in B(2,1)=B(1,1)sB(1,2)s^3B(1,1)s1s\\
&\hskip .2inB(2,2)=B(1,1)sB(1,2)s^3B(1,1)s1s^2\\
\\
K_3=2 &\hskip .2in B(3,1)=B(2,1)s^2B(2,1)s^2B(2,2)s\\
&\hskip .2in       B(3,2)=B(2,1)s^2B(2,1)s^2B(2,2)sB(2,1)s^2B(2,1)s^2B(2,2)\\
\\
K_4= & \hskip.2in  B(4,1)=B(3,1)B(3,2)sB(3,1)B(3,2)s\\
&\hskip .2in       B(4,2)=B(3,1)B(3,2)sB(3,1)B(3,2).
\end{aligned}\en
We assume the remaining levels are the same as level 4.
Each level $n$ satisfies $LDC(n,k)$ for all $k \leq n$.
As such we know that every $k$-factor is finite.
\end{example}

\begin{example}
 There are many examples for which every $k$-factor is isomorphic to the original system. One such example was presented in \cite{AFP} and is isomorphic to an odometer. Another classical example is the Chacon system, pictured in Figure \ref{fig:Chacon}.

\end{example}

\begin{example}
  An example of a system in which every $k$-factor is infinite, but none are isomorphic to the whole system, was mentioned in the Introduction. An infinite entropy, uniquely ergodic, simple, properly ordered Bratteli-Vershik system \tb{(such exist by the Jewett-Krieger \cite{Jewett1970,Krieger1972} and Herman-Putnam-Skau \cite{HPS1992} theorems)} {\em cannot} be isomorphic to any of its $k$-factors.
\end{example}

  We now consider examples when the original system is isomorphic to some, but not all of its $k$-factors.
\begin{example}\label{ex:someiso}
We begin when  only some of the $k$-factors are finite. Since having a finite $k$-factor implies that any level-$k'$ factor with $k'<k$ is also finite, we must have that only finitely many of the $k$-factors are finite, otherwise all would be finite.
We may then give an isomorphic presentation in which {\em no} $k$-factor is finite by telescoping past the last level in which the $k$-factor is finite. Additionally, if you telescope to a level for which the $k$-factor is isomorphic to the original system, you will have a presentation for which every $k$-factor is isomorphic to the original system.
Example \ref{ex:someper} presents a system with a periodic $1$-coding but an aperiodic $2$-coding which determines a symbolic system that is essentially 2-expansive. \tb{(Removing spacer symbols from any level-2 coding leaves either an empty sequence or a sequence in the Prouhet-Thue-Morse system, which is recognizable. See also \cite[Theorem 5.1]{Berthe2017}.)}

\end{example}

\begin{example}
 We consider the case of a system in which each $k$-factor is infinite, and some, but not all, are isomorphic to the original system. If a $k$-factor is isomorphic to the original system, then for all $k'>k$, the level-$k'$ factor is also isomorphic to the original system.
 Therefore, in any such system only the first finitely many $k$-factors can fail to be isomorphic to the original system. In this case there is an equivalent presentation obtained by telescoping past the first level for which the $k$-factor is isomorphic to the full system.
 Then we have a system in which every $k$-factor is isomorphic to the original system. However, the first presentation may be of interest for technical reasons. One simple way to create such an example would be to attach the systems in Figures \ref{fig:Chacon} and \ref{fig:notper}, having them share the spacer reservoir.
 This would satisfy the conditions.
 In \cite{FPS2017} the systems are isomorphic to their level-3 factors, but it is unknown whether they are also isomorphic to their level-1 factors.
\end{example}

\bibliographystyle{mscplain}
\bibliography{PeriodicBiblio2020_03_13}

\end{document}